\documentclass[12pt,twoside]{amsart}
\usepackage{amsmath, amsthm, amscd, amsfonts, amssymb, graphicx}
\usepackage{enumerate}
\usepackage[colorlinks=true,
linkcolor=blue,
urlcolor=cyan,
citecolor=red]{hyperref}
\usepackage{mathrsfs}
\newtheorem{theorem}{Theorem}[section]

\newtheorem{remark}{Remark}[section]

\newtheorem{corollary}{Corollary}[section]

\newtheorem{proposition}{Proposition}[section]
\numberwithin{equation}{section}

\begin{document}
	
\title{A Note on Kantorovich and Ando Inequalities}
\author{Mohammad Sababheh, Hamid Reza Moradi, Ibrahim Halil G\"um\"u\c s, and Shigeru Furuichi}

\begin{abstract}
The main goal of this exposition is to present further analysis of the Kantorovich and Ando operator inequalities. In particular, a new proof of Ando's inequality is given, a new non-trivial refinement of Kantorovich inequality is shown, and some equivalent forms of Kantorovich inequality are presented with a Minkowski-type application.

\end{abstract}
\maketitle
\pagestyle{myheadings}
\markboth{\centerline {A Note on Kantorovich and Ando Inequalities}}
{\centerline {M. Sababheh, H. R. Moradi, I. H. G\"um\"u\c s \& S. Furuichi}}
\bigskip
\bigskip
\section{Introduction and preliminaries}
Let $\mathcal{B}(\mathcal{H})$ be the algebra of bounded linear operators on a complex Hilbert space $\mathcal{H}$, with identity $I_{\mathcal{H}}$ (or $I$ if no confusion arises). For two Hilbert spaces $\mathcal{H}$ and $\mathcal{K}$, a linear mapping $\Phi:\mathcal{B}(\mathcal{H})\to \mathcal{B}(\mathcal{K})$ is said to be positive if it preserves positive operators. That is, if $\Phi(A)\geq 0$ whenever $A\geq 0,$ where an operator $A\in \mathcal{B}(\mathcal{H})$ is said to be positive, denoted $A\geq 0,$ if $\left<Ax,x\right>\geq 0$ for all $x\in \mathcal{H}.$ In addition, if the positive linear mapping $\Phi$ satisfies $\Phi(I_{\mathcal{H}})=I_{\mathcal{K}}$, it is said to be a unital  (or normalized) positive linear mapping. 

Operator convex and operator concave functions have played a major role in understanding the geometry of $\mathcal{B}(\mathcal{H})$. In this context, a function $f:J\to \mathbb{R}$ is said to operator convex if $f((1-t)A+tB)\leq (1-t)f(A)+t f(B)$ for all $0\leq t\leq 1$ and self adjoint operators $A,B$ with spectra in the interval $J$. Operator concave functions are defined similarly. On the other hand, operator monotone functions have a strong relation with operator concave functions. Recall that $f:J\to \mathbb{R}$ is said to be operator monotone if $f(A)\leq f(B)$ for all self-adjoint operators $A, B$ with spectra in the interval $J$, such that $A\leq B.$ Operator monotone decreasing functions are defined similarly.

Unlike scalar monotony and convexity, operator monotony and convexity are strongly related, as stated in the next proposition, which can be found in \cite[Theorem 2.4]{Uchiyama}, and \cite[Theorem 2.1, Theorem 3.1, Theorem 2.3 and Theorem 3.7]{ando1}. 
\begin{proposition}\label{oper_intro_prop}
Let $f:(0,\infty)\to [0,\infty)$ be continuous. Then 
\begin{enumerate}
\item $f$ is operator monotone decreasing if and only if $f$ is operator convex and $f(\infty)<\infty$.
\item $f$ is operator monotone increasing if and only if $f$ is operator concave.
\end{enumerate}
\end{proposition}

The Choi-Davis inequality states that \cite{choi,davis}
\begin{equation}\label{choi_davis_ineq}
f(\Phi(A))\leq \Phi(f(A)),
\end{equation}
for all self adjoint operators $A\in\mathcal{B}(\mathcal{H})$ with spectra in the interval $J$, all operator convex functions $f:J\to\mathbb{R}$ and all unital positive linear mappings $\Phi:\mathcal{B}(\mathcal{H})\to \mathcal{B}(\mathcal{K}).$

In particular, if $A>0$, then 
\begin{equation}\label{inverse_choi_davis}
\Phi(A)^{-1}\leq \Phi(A^{-1}),
\end{equation}
since $f(t)=t^{-1}$ is operator convex on $(0,\infty).$ The inequality \eqref{inverse_choi_davis} can be reversed under the additional condition that $0<mI\leq A\leq MI,$ for some scalars $m,M$ as follows \cite{grueb,mond_houston, 4}

\begin{equation}\label{9}
\Phi \left( {{A}^{-1}} \right)\le \frac{{{\left( M+m \right)}^{2}}}{4Mm}{{\Phi }}\left( A \right)^{-1}.
\end{equation}
Among many other equivalences, we shall prove that \eqref{9} is equivalent to 
\begin{equation}\label{90}
\Phi \left( {{A}^{2}} \right)\le \frac{{{\left( M+m \right)}^{2}}}{4Mm}{{\Phi }}\left( A \right)^{2}.
\end{equation}
Recalling that the geometric mean of two positive invertible operators $A, B$ is defined by
$$A\sharp B=A^{1/2}\left(A^{-1/2}BA^{-1/2}\right)^{1/2}A^{1/2},$$  it is shown in \cite{mond_houston} that \eqref{9} implies
\begin{equation}\label{10}
\Phi \left( {{A}^{-1}} \right)\sharp\Phi \left( A \right)\le \frac{M+m}{2\sqrt{Mm}}.
\end{equation}
The inequality \eqref{9} is usually referred to as the Kantorovich inequality.

In \cite[Lemma 2.1]{3} it is proved that if $f:J\to \mathbb{R}$ is a convex function and $A\in \mathcal{B}\left( \mathcal{H} \right)$ is a self-adjoint operator with  spectrum  in the interval $J$, then for any unital positive linear map $\Phi :\mathcal{B}\left( \mathcal{H} \right)\to \mathcal{B}\left( \mathcal{K} \right)$,
\[f\left( \left\langle \Phi \left( A \right)x,x \right\rangle  \right)\le \left\langle \Phi \left( f\left( A \right) \right)x,x \right\rangle \quad\left( x\in \mathcal{K},\left\| x \right\|=1 \right).\]
As a corollary (see \cite[Theorem 1.4]{5}), we see that if $A$ is a positive operator, then
\begin{equation}\label{1}
{{\left\langle Ax,x \right\rangle }^{r}}\le \left\langle {{A}^{r}}x,x \right\rangle \quad\left( r\ge 1 \right).
\end{equation}
If the operator is positive and invertible, \eqref{1} is also true for $r<0$.

A strongly related inequality that we will discuss is the celebrated Ando's inequality stating \cite{ando_mathias}
\begin{equation}\label{ando_ineq_intro}
\Phi \left( A\sharp B \right)\le \Phi \left( A \right)\sharp\Phi \left( B \right),
\end{equation}
where $A$ and $B$ are two positive operators and $\Phi$ is a unital positive linear map. In addition to the aforementioned references that have been cited, we refer the reader to \cite{li,2micic,01mic,7,mond,hamid_lama,sab_jfs} for further and related discussions.

In this article, we first present a new proof of \eqref{ando_ineq_intro}. This will help better understand this celebrated inequality. Then, we use Kantorovich-type inequalities to provide the reverse of Ando's inequality. Once this is done,  we present a non-trivial refinement of \eqref{10}. Further discussion of the Kantorovich inequality is presented via several equivalent forms. Some applications are given, including a submultiplicative inequality for unital positive linear mappings and an operator Minkowski-type inequality. 

\section{Ando's inequality}
In this section, we first present a new proof of Ando's inequality; then, we prove a reversed version of Ando's inequality. 

Recall that for positive invertible  operators $A$ and $B,$ the Riccati equation $X{{A}^{-1}}X=B$ has the geometric mean $A\sharp B$ as a unique positive solution \cite[Theorem 2.2]{nak}.

 Let 	$X=A\sharp B$ and let $\Phi$ be a unital positive linear map. It follows from Choi's inequality \cite[Proposition 4.3]{choi_assorted},
	$$\Phi \left( X \right)\Phi {{\left( A \right)}^{-1}}\Phi \left( X \right)\le \Phi \left( X{{A}^{-1}}X \right)=\Phi \left( B \right).$$ 
Therefore,	 
	$${{\left( \Phi {{\left( A \right)}^{-\frac{1}{2}}}\Phi \left( X \right)\Phi {{\left( A \right)}^{-\frac{1}{2}}} \right)}^{2}}\le \Phi {{\left( A \right)}^{-\frac{1}{2}}}\Phi \left( B \right)\Phi {{\left( A \right)}^{-\frac{1}{2}}}.$$ 
Since $f\left( t \right)={{t}^{\frac{1}{2}}}$ is  operator monotone \cite[Corollary 1.16]{5}, we get
	$$\Phi {{\left( A \right)}^{-\frac{1}{2}}}\Phi \left( X \right)\Phi {{\left( A \right)}^{-\frac{1}{2}}}\le {{\left( \Phi {{\left( A \right)}^{-\frac{1}{2}}}\Phi \left( B \right)\Phi {{\left( A \right)}^{-\frac{1}{2}}} \right)}^{\frac{1}{2}}}.$$ 
Consequently,
$$\Phi \left( X \right)\le \Phi {{\left( A \right)}^{\frac{1}{2}}}{{\left( \Phi {{\left( A \right)}^{-\frac{1}{2}}}\Phi \left( B \right)\Phi {{\left( A \right)}^{-\frac{1}{2}}} \right)}^{\frac{1}{2}}}\Phi {{\left( A \right)}^{\frac{1}{2}}},$$
which is equivalent to 
	$$\Phi \left( A\sharp B \right)\le \Phi \left( A \right)\sharp\Phi \left( B \right).$$
This proves Ando's inequality.

Next, we utilize \eqref{90} to prove a reversed version of Ando's inequality under the sandwich condition. We remark that this reversed version has been shown in \cite[Theorem 4]{lee} using a completely different method. In this article, we utilize the Kantorovich-type inequalities to offer this version. This helps understand the relation between Ando-type and Kantorovich-type inequalities.
\begin{proposition}\label{prop_nerev}
Let $\Phi :\mathcal{B}\left( \mathcal{H} \right)\to \mathcal{B}\left( \mathcal{K} \right)$ be a unital positive linear mapping and let $A,B\in \mathcal{B}(\mathcal{H})$ be positive operators such that $m^2A\leq B\leq M^2A,$ for some positive scalars $m,M.$ Then
$$\Phi(A)\sharp \Phi(B) \leq \frac{M+m}{2\sqrt{mM}}\Phi(A\sharp B).$$ 
\end{proposition}
\begin{proof}
For the given $\Phi$ and $A$, define the positive unital linear mapping $\Psi$ by
 $\Psi \left( X \right)\equiv \Phi\left( A \right)^{-\frac{1}{2}}\Phi \left( {{A}^{\frac{1}{2}}}X{{A}^{\frac{1}{2}}} \right)\Phi\left( A \right)^{-\frac{1}{2}}$ and let $C=\left( A^{-\frac{1}{2}}BA^{-\frac{1}{2}}\right)^{\frac{1}{2}}.$ Since $m^2A\leq B\leq M^2A,$ it follows that $mI\leq C\leq MI.$ Therefore, we may apply the inequality
$\Psi(C^2)\leq \left(\frac{M+m}{2\sqrt{Mm}}\right)^2\Psi(C)^2$ to obtain
$$\Phi(A)^{-\frac{1}{2}}\Phi(B)\Phi(A)^{-\frac{1}{2}}\leq \left(\frac{M+m}{2\sqrt{Mm}}\right)^2 \left(\Phi(A)^{-\frac{1}{2}}\Phi(A\sharp B)\Phi(A)^{-\frac{1}{2}}\right)^2.$$ Since the function $f(t)=t^{\frac{1}{2}}$ is operator monotone, it follows that
$$\left(\Phi(A)^{-\frac{1}{2}}\Phi(B)\Phi(A)^{-\frac{1}{2}}\right)^{\frac{1}{2}}\leq \frac{M+m}{2\sqrt{Mm}} \Phi(A)^{-\frac{1}{2}}\Phi(A\sharp B)\Phi(A)^{-\frac{1}{2}},$$ which implies the desired inequality.
\end{proof}

In fact, Ando's inequality follows from a more general result that
\begin{equation}\label{ando_gen}
\Phi(A\sigma_f B)\leq \Phi(A)\sigma_f \Phi(B),
\end{equation}
where $A,B$ are positive and $\sigma_f$ is an operator mean with representing function $f$.  In the next result, we show that if $f$ is operator convex, then Ando's inequality is reversed, then we show that this reversed Ando inequality implies \eqref{inverse_choi_davis}.  We point out here that Theorem \ref{thm_rev_and} does not follow from \eqref{ando_gen}, as $f$ is a positive function. So, multiplying \eqref{ando_gen} with -1 does not imply Theorem  \ref{thm_rev_and}.
\begin{theorem}\label{thm_rev_and}
Let $A,B\in\mathcal{B}(\mathcal{H})$ be positive invertible, $f:(0,\infty)\to (0,\infty)$ be a given operator convex function and $\Phi:\mathcal{B}(\mathcal{H})\to \mathcal{B}(\mathcal{K})$ be a positive unital linear mapping. Then
$$\Phi(A\sigma_f B)\geq \Phi(A)\sigma_f \Phi(B),$$ where the connection $\sigma_f$ is defined by
$$A\sigma_fB=A^{\frac{1}{2}}f\left(A^{-\frac{1}{2}}BA^{-\frac{1}{2}}\right)A^{\frac{1}{2}}.$$
\end{theorem}
\begin{proof}
For the given parameters, define
$$\Psi \left( X \right)\equiv{{\Phi }}\left( A \right)^{-\frac{1}{2}}\Phi \left( {{A}^{\frac{1}{2}}}X{{A}^{\frac{1}{2}}} \right){{\Phi }}\left( A \right)^{-\frac{1}{2}};\;X\in\mathcal{B}(\mathcal{H}).$$ Then $\Psi$ is positive unital. Since $f$ is operator convex, \eqref{choi_davis_ineq} implies $f\left(\Psi(X)\right)\leq \Psi(f(X)),$ for any self adjoint $X\in\mathcal{B}(\mathcal{H}).$ Let $X=A^{-\frac{1}{2}}BA^{-\frac{1}{2}}$ and apply this latter inequality to get
$$f\left(\Phi(A)^{-\frac{1}{2}}\Phi(B)\Phi(A)^{-\frac{1}{2}}\right)\leq {{\Phi }}\left( A \right)^{-\frac{1}{2}}\Phi \left( {{A}^{\frac{1}{2}}}f\left( A^{-\frac{1}{2}}BA^{-\frac{1}{2}}  \right){{A}^{\frac{1}{2}}} \right){{\Phi }}\left( A \right)^{-\frac{1}{2}},$$ which is equivalent to
$$ \Phi(A)\sigma_f \Phi(B)\leq\Phi(A\sigma_f B),$$ as desired.
\end{proof}
Interestingly, Theorem \ref{thm_rev_and} implies \eqref{inverse_choi_davis}, as follows: In Theorem \ref{thm_rev_and}, let $f(t)=t^2$ and $B=I.$ Since  $f$ is operator convex, we may apply the theorem. Direct computations show that $A\sigma_fB=A^{-1}.$ Consequently,
$$\Phi(A^{-1})= \Phi(A\sigma_fB)\geq \Phi(A)\sigma_f \Phi(B)=\Phi(A)^{-1};$$ as desired.

\begin{remark}
We remark that in Theorem \ref{thm_rev_and}, if we let $A=I$, we get $$f(\Phi(B))\leq \Phi(f(B));$$ an inequality that is equivalent to the fact that $f$ is operator convex. This shows that the inequality in Theorem \ref{thm_rev_and} is equivalent to the fact that $f$ is operator convex.
\end{remark}

\section{Further analysis of the Kantorovich inequality}
This section is devoted to the study of Kantorovich inequality \eqref{9}, where we begin by giving multiple equivalent statements. It should be remarked that these individual statements are well known, but their equivalence is the aim of  Theorem \ref{2}.

We will use the following observation to prove the next result. Let $\Phi$ be a given unital positive linear map and let $\Phi'$ be another unital positive linear map. Then \eqref{ando_ineq_intro} implies 
\[\Phi '\left( \Phi \left( A\sharp B \right) \right)\le \Phi '\left( \Phi \left( A \right)\sharp\Phi \left( B \right) \right)\le \Phi '\left( \Phi \left( A \right) \right)\sharp\Phi '\left( \Phi \left( B \right) \right).\]
Defining $\Phi '\left( T \right)\equiv \left\langle Tx,x \right\rangle$, with $x\in \mathcal{H}$, $\left\| x \right\|=1$, we get
\begin{equation}\label{17}
\begin{aligned}
   \left\langle \Phi \left( A\sharp B \right)x,x \right\rangle &\le \left\langle \Phi \left( A \right)\sharp\Phi \left( B \right)x,x \right\rangle  \\ 
 & \le \left\langle \Phi \left( A \right)x,x \right\rangle \sharp\left\langle \Phi \left( B \right)x,x \right\rangle  \\ 
 & =\sqrt{\left\langle \Phi \left( A \right)x,x \right\rangle \left\langle \Phi \left( B \right)x,x \right\rangle }.  
\end{aligned}
\end{equation}

As we mentioned earlier, these statements are all known. We have already seen (i), (iii) and (iv) (from Theorem \ref{2}) in \eqref{9}, \eqref{10} and \eqref{90} respectively. For (ii) in Theorem \ref{2}, it can be easily deduced from Proposition \ref{prop_nerev} on letting $B=A^{-1}$ and defining the new mapping $\Phi'(X)=\left<\Phi(X)x,x\right>.$
\begin{theorem}\label{2}
Let $A\in \mathcal{B}\left( \mathcal{H} \right)$ satisfying $mI\le A\le MI$ for some scalars $0<m<M$. Then the following assertions are equivalent. 
\begin{itemize}
	\item[(i)] $\Phi \left( {{A}^{-1}} \right)\le {{\left( \frac{M+m}{2\sqrt{Mm}} \right)}^{2}}{{\Phi }}\left( A \right)^{-1}$ for any unital positive linear mapping $\Phi :\mathcal{B}\left( \mathcal{H} \right)\to \mathcal{B}\left( \mathcal{K} \right)$.
	\item[(ii)] $\left\langle \Phi \left( {{A}^{-1}} \right)x,x \right\rangle \le {{\left( \frac{M+m}{2\sqrt{Mm}} \right)}^{2}}{{\left\langle \Phi \left( A \right)x,x \right\rangle }^{-1}}$ for any unit vector $x \in \mathcal{K}$ and any positive unital linear mapping $\Phi :\mathcal{B}\left( \mathcal{H} \right)\to \mathcal{B}\left( \mathcal{K} \right)$.
	\item[(iii)] $\Phi \left( {{A}^{-1}} \right)\sharp\Phi \left( A \right)\le \frac{M+m}{2\sqrt{Mm}}I$ for any unital positive linear mapping $\Phi :\mathcal{B}\left( \mathcal{H} \right)\to \mathcal{B}\left( \mathcal{K} \right)$.
	
\item[(iv)] $\Phi \left( {{A}^{2}} \right)\le {{\left( \frac{M+m}{2\sqrt{Mm}} \right)}^{2}}{{\Phi }}\left( A \right)^{2}$ for any unital positive linear mapping $\Phi :\mathcal{B}\left( \mathcal{H} \right)\to \mathcal{B}\left( \mathcal{K} \right)$.

\end{itemize}
\end{theorem}

\begin{proof}
$\left( ii \right)\Rightarrow \left( i \right)$\\ Assuming that $(ii)$ is true. Using \eqref{1} we see that for any unit vector $x\in \mathcal{K}$,
\[\begin{aligned}
\left\langle \Phi \left( {{A}^{-1}} \right)x,x \right\rangle &\le {{\left( \frac{M+m}{2\sqrt{Mm}} \right)}^{2}}{{\left\langle \Phi \left( A \right)x,x \right\rangle }^{-1}} \\ 
& \le {{\left( \frac{M+m}{2\sqrt{Mm}} \right)}^{2}}\left\langle {{\Phi }}\left( A \right)^{-1}x,x \right\rangle.   
\end{aligned}\]
This implies the desired result.\\
$\left( ii \right)\Rightarrow \left( iii \right)$\\
Since
	\[\left\langle \Phi \left( {{A}^{-1}} \right)x,x \right\rangle \le {{\left( \frac{M+m}{2\sqrt{Mm}} \right)}^{2}}{{\left\langle \Phi \left( A \right)x,x \right\rangle }^{-1}},\] 
we get
\[\sqrt{\left\langle \Phi \left( {{A}^{-1}} \right)x,x \right\rangle \left\langle \Phi \left( A \right)x,x \right\rangle }\le \frac{M+m}{2\sqrt{Mm}}.\] 
On the other hand, by \eqref{17},
\[\left\langle \Phi \left( {{A}^{-1}} \right)\sharp\Phi \left( A \right)x,x \right\rangle \le \sqrt{\left\langle \Phi \left( {{A}^{-1}} \right)x,x \right\rangle \left\langle \Phi \left( A \right)x,x \right\rangle }\] 
we get for any unit vector $x\in \mathcal{K}$,
\[\left\langle \Phi \left( {{A}^{-1}} \right)\sharp\Phi \left( A \right)x,x \right\rangle \le \frac{M+m}{2\sqrt{Mm}}.\] 

$\left( i \right)\Rightarrow \left( iii \right)$\\
As shown in \cite{4}, $(i)$ implies $(iii)$, but here we give another proof. For positive linear functional $ \Psi :  \mathcal{B}(\mathcal{H}) \to \mathbb{R}^{+}$ defined by $\Psi(A) := \langle \Phi(A)x,x\rangle$ (with the understanding $\Phi :\mathcal{B}\left( \mathcal{H} \right)\to \mathcal{B}\left( \mathcal{K} \right)$ is unital positive linear map and $x \in \mathcal{K}$ is a unit vector), $(i)$ implies
\[\left\langle \Phi \left( {{A}^{-1}} \right)x,x \right\rangle \left\langle \Phi \left( A \right)x,x \right\rangle \le {{\left( \frac{M+m}{2\sqrt{Mm}} \right)}^{2}}.\]
This yields
\[\sqrt{\left\langle \Phi \left( {{A}^{-1}} \right)x,x \right\rangle \left\langle \Phi \left( A \right)x,x \right\rangle }\le \frac{M+m}{2\sqrt{Mm}}.\]
By \eqref{17},  we have
\[\left\langle \Phi \left( {{A}^{-1}} \right)\sharp\Phi \left( A \right)x,x \right\rangle \le \frac{M+m}{2\sqrt{Mm}},\]
as desired.

$\left( iii \right)\Rightarrow \left( ii \right)$\\
We may take the unital positive linear map $\Psi :\mathcal{B}\left( \mathcal{K} \right)\to {{\mathbb{R}}^{+}}$ defined by $\Psi(A) := \langle \Phi(A)x,x\rangle$ for any unital positive map $\Phi$, $A>0$ and a unit vector $x \in \mathcal{K}$.
From the assumption $(iii)$ with $\Psi$, we have for any unit vector $x \in \mathcal{K}$,
$$
\sqrt{\langle \Phi(A)x,x\rangle \langle \Phi(A^{-1})x,x\rangle }=\Psi(A)\sharp\Psi(A^{-1}) \leq \frac{M+m}{2\sqrt{Mm}},
$$
which implies $(ii)$:
$$
\langle \Phi(A^{-1})x,x\rangle \leq \left(\frac{M+m}{2\sqrt{Mm}}\right)^2\langle \Phi(A)x,x\rangle^{-1}.
$$

$(i)\Rightarrow (iv)$\\
By taking $\Psi \left( X \right)\equiv \Phi\left( A \right)^{-\frac{1}{2}}\Phi \left( {{A}^{\frac{1}{2}}}X{{A}^{\frac{1}{2}}} \right){{\Phi }}\left( A \right)^{-\frac{1}{2}}$, where $\Phi $ is an arbitrary unital positive linear map in $(i)$,  we obtain
\[\begin{aligned}
   {{\Phi }}\left( A \right)^{-1}&\le {{\left( \frac{M+m}{2\sqrt{Mm}} \right)}^{2}}{{\left( {{\Phi }}\left( A \right)^{-\frac{1}{2}}\Phi \left( {{A}^{2}} \right){{\Phi }}\left( A \right)^{-\frac{1}{2}} \right)}^{-1}} \\ 
 & ={{\left( \frac{M+m}{2\sqrt{Mm}} \right)}^{2}}{{\Phi }}\left( A \right)^{\frac{1}{2}}{{\Phi }}\left( {{A}^{2}} \right)^{-1}{{\Phi }}\left( A \right)^{\frac{1}{2}}.
\end{aligned}\]
This implies
\[{{\Phi }}\left( A \right)^{-2}\le {{\left( \frac{M+m}{2\sqrt{Mm}} \right)}^{2}}{{\Phi }}\left( {{A}^{2}} \right)^{-1}.\]
By taking the inverse, we infer
\[\Phi \left( {{A}^{2}} \right)\le {{\left( \frac{M+m}{2\sqrt{Mm}} \right)}^{2}}{{\Phi }}\left( A \right)^{2}.\]

$(iv)\Rightarrow (i)$\\
Assuming $(iv)$ and replacing $A$ with $A^{-1}$, we obtain
$$\Phi((A^{-1})^2)\leq {{\left( \frac{M+m}{2\sqrt{Mm}} \right)}^{2}} \Phi(A^{-1})^2,$$ for any unital positive linear mapping $\Phi.$ Again, defining 
$\Psi \left( X \right)\equiv{{\Phi }}\left( A \right)^{-\frac{1}{2}}\Phi \left( {{A}^{\frac{1}{2}}}X{{A}^{\frac{1}{2}}} \right){{\Phi }}\left( A \right)^{-\frac{1}{2}}$ and applying this latter inequality to $\Psi$, we obtain
$$\Psi((A^{-1})^2)\leq {{\left( \frac{M+m}{2\sqrt{Mm}} \right)}^{2}} \Psi(A^{-1})^2,$$ which is equivalent to
$$\Phi(A)^{-\frac{1}{2}}\Phi(A^{-1}) \Phi(A)^{-\frac{1}{2}} \leq  {{\left( \frac{M+m}{2\sqrt{Mm}} \right)}^{2}} \Phi(A)^{-2}. $$ This implies
$$\Phi(A^{-1})\leq {{\left( \frac{M+m}{2\sqrt{Mm}} \right)}^{2}} \Phi(A)^{-1},$$ as required. This completes the proof.

\end{proof}

We have seen that the Kantorovich inequality \eqref{9} is equivalent to the inequality
\begin{equation}\label{1_new}
\Phi \left( {{A}^{-1}} \right)\sharp\Phi \left( A \right)\le \frac{M+m}{2\sqrt{Mm}}I,
\end{equation}
which in turn implies
\[\Phi \left( {{A}^{-1}} \right)\sharp\Phi \left( A \right)\le \left\| \Phi \left( {{A}^{-1}} \right)\sharp\Phi \left( A \right) \right\|I\le \frac{M+m}{2\sqrt{Mm}}I.\]
Next, we present a more precise estimate than  \eqref{1_new}, as follows. 
\begin{theorem}
Let $A\in\mathcal{B}(\mathcal{H})$ be a positive operator such that $mI\leq A\leq MI$, for some positive scalars $m,M$. If $\Phi:\mathcal{B}(\mathcal{H})\to \mathcal{B}(\mathcal{K})$ is a positive unital linear mapping, then
\[\Phi \left( {{A}^{-1}} \right)\sharp\Phi \left( A \right)\le \left\| {{\left( \Phi {{\left( A \right)}^{\frac{1}{2}}}\Phi \left( {{A}^{-1}} \right)\Phi {{\left( A \right)}^{\frac{1}{2}}} \right)}^{\frac{1}{2}}} \right\|I\le \frac{M+m}{2\sqrt{Mm}}I.\]
\end{theorem}
\begin{proof}

Kantorovich inequality states that if $mI\le A\le MI$ and $\Phi$ is unital positive linear map, then
\[\Phi \left( {{A}^{-1}} \right)\le \frac{{{\left( M+m \right)}^{2}}}{4Mm}\Phi {{\left( A \right)}^{-1}}.\]
This implies
\[\Phi {{\left( A \right)}^{\frac{1}{2}}}\Phi \left( {{A}^{-1}} \right)\Phi {{\left( A \right)}^{\frac{1}{2}}}\le \frac{{{\left( M+m \right)}^{2}}}{4Mm}I.\]
Noting operator monotony of the function $f(t)=t^{\frac{1}{2}},$ we have
\[{{\left( \Phi {{\left( A \right)}^{\frac{1}{2}}}\Phi \left( {{A}^{-1}} \right)\Phi {{\left( A \right)}^{\frac{1}{2}}} \right)}^{\frac{1}{2}}}\le \frac{M+m}{2\sqrt{Mm}}I.\]
Whence
\[\left\| {{\left( \Phi {{\left( A \right)}^{\frac{1}{2}}}\Phi \left( {{A}^{-1}} \right)\Phi {{\left( A \right)}^{\frac{1}{2}}} \right)}^{\frac{1}{2}}} \right\|\le \frac{M+m}{2\sqrt{Mm}}.\]
On the other hand, from \cite[Corollary 2.13]{math_auj}, we infer that
\[\begin{aligned}
   \left\| \Phi \left( {{A}^{-1}} \right)\sharp\Phi \left( A \right) \right\|&\le \left\| {{\left( \Phi {{\left( A \right)}^{\frac{1}{2}}}\Phi \left( {{A}^{-1}} \right)\Phi {{\left( A \right)}^{\frac{1}{2}}} \right)}^{\frac{1}{2}}} \right\|.  
\end{aligned}\]
Consequently,
	\[\Phi \left( {{A}^{-1}} \right)\sharp\Phi \left( A \right)\le \left\| {{\left( \Phi {{\left( A \right)}^{\frac{1}{2}}}\Phi \left( {{A}^{-1}} \right)\Phi {{\left( A \right)}^{\frac{1}{2}}} \right)}^{\frac{1}{2}}} \right\|I\le \frac{M+m}{2\sqrt{Mm}}I.\] This completes the proof.
\end{proof}

In what follows, we present a reversed version of \cite[Proposition 4.3]{choi_assorted} using $(iv)$ in Theorem \ref{2}. We remark that this proposition has already been shown in \cite[Corollary 3.11]{7}, using a different technique.
\begin{proposition}
Let $\Phi :\mathcal{B}\left( \mathcal{H} \right)\to \mathcal{B}\left( \mathcal{K} \right)$ be a unital positive linear mapping and let $A,B\in \mathcal{B}(\mathcal{H})$ be such that $mA\leq B\leq MA,$ for some scalars $m,M.$ Then $$\Phi(BA^{-1}B)\leq {{\left( \frac{M+m}{2\sqrt{Mm}} \right)}^{2}} \Phi(B)\Phi(A)^{-1}\Phi(B).$$
\end{proposition}
\begin{proof}
From Theorem \ref{2}, we have  $\Phi \left( {{A}^{2}} \right)\le {{\left( \frac{M+m}{2\sqrt{Mm}} \right)}^{2}}{{\Phi }}\left( A \right)^{2}$ for any $\Phi$ and $mI\leq A\leq MI.$ When $mA\leq B\leq MA,$ we get $mI\leq A^{-\frac{1}{2}}BA^{-\frac{1}{2}}\leq MI.$ Therefore, any positive unital linear $\Phi$ satisfies
$$\Phi\left((A^{-\frac{1}{2}}BA^{-\frac{1}{2}})^2\right)\leq {{\left( \frac{M+m}{2\sqrt{Mm}} \right)}^{2}} \Phi\left(A^{-\frac{1}{2}}BA^{-\frac{1}{2}}\right)^2.$$
Letting $\Psi \left( X \right)\equiv \Phi\left( A \right)^{-\frac{1}{2}}\Phi \left( {{A}^{\frac{1}{2}}}X{{A}^{\frac{1}{2}}} \right)\Phi\left( A \right)^{-\frac{1}{2}}$ and applying the latter inequality for $\Psi$, we obtain
$$\Phi(A)^{-\frac{1}{2}}\Phi(BA^{-1}B)\Phi(A)^{-\frac{1}{2}}\leq  {{\left( \frac{M+m}{2\sqrt{Mm}} \right)}^{2}} \Phi(A)^{-\frac{1}{2}} \Phi(B) \Phi(A)^{-1}\Phi(B) \Phi(A)^{-\frac{1}{2}},$$ which implies the desired inequality.
\end{proof}

We notice that $(ii)$ in Theorem \ref{2} is a particular case of the following more general result, whose proof is an implementation of the well-known Mond-Pe\v{c}ari\'{c} method. We remark that this theorem follows from \cite[Theorem 2.2]{7} upon letting $\Psi(X)=\left<\Phi(X)x,x\right>$, but we present the proof here for the reader's convenience. 
\begin{theorem}\label{lemma1.1}
	Let $A \in \mathcal{B}(\mathcal{H})$ be a self-adjoint operator with the spectra in the interval $[m,M]$ and let $\Phi$ be a unital positive linear mapping on $\mathcal{B}(\mathcal{H})$. If $f:\left[ m,M \right]\to \mathbb{R}$ is a convex function, then for any unit vector $x \in \mathcal{H}$ and $\alpha \ge 0$
	\[\left\langle \Phi \left( f\left( A \right) \right)x,x \right\rangle \le \beta +\alpha f\left( \left\langle \Phi \left( A \right)x,x \right\rangle  \right)\]
	holds, where $\beta ={{\max }_{m\le t\le M}}\left\{ {{a}_{f}}t+{{b}_{f}}-\alpha f\left( t \right) \right\}$ with ${{a}_{f}}={\left( f\left( M \right)-f\left( m \right) \right)}/{\left( M-m \right)}\;$ and ${{b}_{f}}={\left( Mf\left( m \right)-Mf\left( m \right) \right)}/{\left( M-m \right)}\;$.
\end{theorem}
\begin{proof}
	Since $f$ is convex on $[m,M]$, we have for any $m\le t\le M$,
	\[f\left( t \right)\le {{a}_{f}}t+{{b}_{f}}.\] 
	It follows from the continuous functional calculus that
	\[f\left( A \right)\le {{a}_{f}}A+{{b}_{f}}I.\] 
	The assumptions on $\Phi$ implies
	\[\Phi \left( f\left( A \right) \right)\le {{a}_{f}}\Phi \left( A \right)+{{b}_{f}}I.\] 
	Thus, for any unit vector $x \in \mathcal{H}$
	\[\left\langle \Phi \left( f\left( A \right) \right)x,x \right\rangle \le {{a}_{f}}\left\langle \Phi \left( A \right)x,x \right\rangle +{{b}_{f}}.\] 
	Therefore,
	\[\begin{aligned}
	\left\langle \Phi \left( f\left( A \right) \right)x,x \right\rangle -\alpha f\left( \left\langle \Phi \left( A \right)x,x \right\rangle  \right)&\le {{a}_{f}}\left\langle \Phi \left( A \right)x,x \right\rangle +{{b}_{f}}-\alpha f\left( \left\langle \Phi \left( A \right)x,x \right\rangle  \right) \\ 
	& \le \underset{m\le t\le M}{\mathop{\max }}\,\left\{ {{a}_{f}}t+{{b}_{f}}-\alpha f\left( t \right) \right\} \\ 
	& =\beta .  
	\end{aligned}\] 
\end{proof}

\begin{corollary}
Let $A\in \mathcal{B}\left( \mathcal{H} \right)$ be a positive and invertible operator satisfying $mI\le A\le MI$ for some scalars $0<m<M$ and $\Phi :\mathcal{B}\left( \mathcal{H} \right)\to \mathcal{B}\left( \mathcal{K} \right)$ be a unital positive linear map.  Then, for any unit vector $x \in \mathcal{K}$,
\begin{equation}\label{5}
\left\langle \Phi \left( A^{p} \right)x,x \right\rangle \le
K(p,m,M)
\left\langle \Phi \left( A \right)x,x \right\rangle ^p
\end{equation}
where $ K(p,m,M)$ is the generalized Kantorovich constant defined by
\begin{equation}\label{18}
K(p,m,M) :=\frac{(m M^p-Mm^p)}{(p-1)(M-m)}\left(\frac{(p-1)}{p} \frac{(M^p-m^p)}{(m M^p-Mm^p)}\right)^p.
\end{equation}
\end{corollary}
\begin{proof}
If we take $f(t) = t^p$, $(t>0)$, for $p\geq 1$ or $p \leq 0$, we obtain \eqref{5}.
\end{proof}

\begin{remark}
 We know that if $A$ is a positive operator, then for any $p\ge 1$
\begin{equation}\label{3_old}
\Phi \left( {{A}^{p}} \right)\le K\left( p,m,M \right){{\Phi }}\left( A \right)^{p}.
\end{equation}
If the operator $A$ is positive and invertible, \eqref{3_old} is also true for $p<0$. Evidently, \eqref{3_old} implies
\begin{equation}\label{6}
\left\langle \Phi \left( {{A}^{p}} \right)x,x \right\rangle \le K\left( p,m,M \right)\left\langle {{\Phi }}\left( A \right)^{p}x,x \right\rangle
\end{equation}
for any unit vector $x \in \mathcal{K}$. Thus, \eqref{5} can be considered as an improvement of \eqref{6}, thanks to \eqref{1}.
\end{remark}

Notice that the case $p=2$ in \eqref{3_old} reduces to
\begin{equation}\label{03}
\Phi \left( {{A}^{2}} \right)\le {{\left( \frac{M+m}{2\sqrt{Mm}} \right)}^{2}}{{\Phi }}\left( A \right)^{2}.
\end{equation}

\section{Related results via operator convex and operator monotone functions}
An additive form (see \cite[Theorem 2]{6}) of \eqref{03} is incorporated in
\begin{equation}\label{12}
{{\Phi }}\left( {{A}^{2}} \right)^{\frac{1}{2}}\le \frac{{{\left( M-m \right)}^{2}}}{4\left( M+m \right)}+\Phi \left( A \right).
\end{equation} 
In this section we present a two-term version of this inequality in a more general setting; where this inequality is looked at as $f^{-1}(\Phi(f(A))$ where $f(t)=t^2.$ Then, we present a Minkowski-type inequality for tuples of operators. \\
For the used notation in the next theorem, we shall adopt the following notations
$$\alpha[f;m,M]=\max \left\{ \frac{1}{f\left( t \right)}\left( \frac{f\left( M \right)-f\left( m \right)}{M-m}t+\frac{Mf\left( m \right)-mf\left( M \right)}{M-m} \right):\text{ }m\le t\le M \right\},$$
and

\begin{equation}\label{betanode_formula}
\beta_0[f;m,M]=\max \left\{ f(t)- \frac{f\left( M \right)-f\left( m \right)}{M-m}t-\frac{Mf\left( m \right)-mf\left( M \right)}{M-m} :\text{ }m\le t\le M \right\},
\end{equation}
 where $f:[m,M]\to (0,\infty)$ is a given function.

\begin{theorem}\label{general_mink_two_terms}
Let $A,B\in \mathcal{B}\left( \mathcal{H} \right)$ be two positive operators satisfying $mI\le A,B\le MI$ for some scalars $0<m<M$ and let $\Phi$ be a unital positive linear mapping on $\mathcal{B}\left( \mathcal{H} \right)$. If $f:(0,\infty)\to (0,\infty)$ is a 1-1  operator convex function such that $f^{-1}$ is operator monotone, then
\begin{equation*}
f^{-1}(\Phi(f(A))+f^{-1}(\Phi(f(B))\leq \alpha f^{-1}(\Phi(f(A+B)),
\end{equation*}
and
\begin{equation*}
f^{-1}(\Phi(f(A))+f^{-1}(\Phi(f(B))\leq \beta+ f^{-1}(\Phi(f(A+B)),
\end{equation*}
where  $\alpha = \alpha[f;m,M], m'=\min_{t \in[m,M]}f(t), M'=\max_{t \in[m,M]}f(t)$ and $\beta=2\beta_0[f^{-1};m',M'].$
\end{theorem}
\begin{proof}
From the Choi-Davis inequality, we have
$$f(\Phi(A+B))\leq \Phi(f(A+B)),$$ which then implies
\begin{equation}\label{needed_1}
\Phi(A+B)\leq f^{-1}(\Phi(f(A+B)),
\end{equation}
since, by assumption, $f^{-1}$ is operator monotone. Furthermore, by \cite[Corollary 2.5]{7}, we have 
\begin{equation}\label{needed_f_phi_2}
\Phi \left( f\left( A \right) \right)\le \alpha f\left( \Phi \left( A \right) \right)\text{ and }\Phi \left( f\left( B \right) \right)\le \alpha f\left( \Phi \left( B \right) \right)
\end{equation}
since $f$ is convex.  Now since $f^{-1}$ is operator monotone, $\alpha>1$, the latter inequalities imply
\begin{equation}\label{needed_2}
{{f}^{-1}}\left( \Phi \left( f\left( A \right) \right) \right)\le \alpha \Phi \left( A \right)\text{ and }{{f}^{-1}}\left( \Phi \left( f\left( B \right) \right) \right)\le \alpha \Phi \left( B \right).
\end{equation}
Combining \eqref{needed_2} and \eqref{needed_1} imply
\begin{align*}
f^{-1}(\Phi(f(A))+f^{-1}(\Phi(f(B))&\leq \alpha(\Phi(A)+\Phi(B))\\
&=\alpha \Phi(A+B)\\
&\leq   \alpha f^{-1}(\Phi(f(A+B)).
\end{align*}
This proves the first inequality. To prove the second inequality, recall that if $g$ is operator concave then $g(\Phi(A))\geq \Phi(g(A)).$ Further, we know, from \cite[Remark 2.3]{7}, that if $mI\leq A\leq MI,$ then
\begin{equation}\label{needed_phi_g_1}
g(\Phi(A))\leq \beta_0[g;m,M]+\Phi(g(A)).
\end{equation}
By assumption, $f^{-1}:(0,\infty)\to (0,\infty)$ is operator monotone, hence it is operator concave (Proposition \ref{oper_intro_prop}). So, applying \eqref{needed_phi_g_1}  with $g=f^{-1},$ we obtain
\begin{equation}\label{needed_finversef}
f^{-1}(\Phi(f(A))\leq \beta_0[f^{-1};m',M']+\Phi(A)\quad{\text{and}}\quad f^{-1}(\Phi(f(B))\leq \beta_0[f^{-1};m',M']+\Phi(B),
\end{equation}
where $m'=\min\{f(t):m\leq t\leq M\}$ and $M'=\max\{f(t):m\leq t\leq M\}.$ Adding the two inequalities we get
$$f^{-1}(\Phi(f(A))+f^{-1}(\Phi(f(B))\leq 2\beta_0[f^{-1};m',M']+\Phi(A+B).$$ But we know that $$\Phi(A+B)=\Phi(f^{-1}(f(A+B)))\leq f^{-1}(\Phi(f(A+B)))$$ since $f^{-1}$ is operator concave. Thus, we have shown that 
$$f^{-1}(\Phi(f(A))+f^{-1}(\Phi(f(B))\leq \beta+ f^{-1}(\Phi(f(A+B))),$$ where $\beta=2\beta_0[f^{-1};m',M'].$ This completes the proof.
\end{proof}

Notice that if $f(t)=t^2$, then $m'=m^2, M'=M^2$ and $f^{-1}(t)=\sqrt{t}.$ Calculating the maximum in \eqref{betanode_formula}, we obtain
$$\beta_0[f^{-1};m',M']=\frac{(M-m)^2}{4(m+M)}.$$ Therefore, the inequality \eqref{12} follows from \eqref{needed_finversef}.

In general, if $f(t)=t^{p}, p\geq 1,$ we can show that
$$
\beta_0[f^{-1};m',M']=g_p\left(\left\{p\frac{M-m}{M^p-m^p}\right\}^{\frac{p}{1-p}}\right),
$$
where $$g_p(t)=t^{1/p}-\frac{M-m}{M^p-m^p}t-\frac{mM^p-Mm^p}{M^p-m^p}.$$ We will use the notation 
\begin{equation}\label{beta_p}
\beta_p=2\beta_0[t^{1/p};m^p,M^p].
\end{equation}
\begin{remark}\label{power_remark}
Tracing the proof of Theorem \ref{general_mink_two_terms}, one can see that if $f^{-1}$ is a power function, then \eqref{needed_f_phi_2} implies

\[{{f}^{-1}}\left( \Phi \left( f\left( A \right) \right) \right)\le f^{-1}(\alpha) \Phi \left( A \right)\text{ and }{{f}^{-1}}\left( \Phi \left( f\left( B \right) \right) \right)\le f^{-1}(\alpha) \Phi \left( B \right).
\] This implies
\begin{equation*}\label{power_function_remark}
f^{-1}(\Phi(f(A))+f^{-1}(\Phi(f(B))\leq f^{-1}(\alpha) f^{-1}(\Phi(f(A+B)).
\end{equation*}
\end{remark}
In particular, letting $f(t)=t^{p}, 1\leq p\leq 2$ in Remark \ref{power_remark}, we obtain the following special cases. We refer the reader to \cite{kian} for a detailed discussion of the next corollary.
\begin{corollary}\label{4}
Let $A,B\in \mathcal{B}\left( \mathcal{H} \right)$ be two self-adjoint operators satisfying $mI\le A,B\le MI$ for some scalars $0<m<M$ and let $\Phi$ be a unital positive linear mapping on $\mathcal{B}\left( \mathcal{H} \right)$. If $1\leq p\leq 2,$ then 
\begin{equation*}
{{\Phi }}\left( {{A}^{p}} \right)^{\frac{1}{p}}+{{\Phi }}\left( {{B}^{p}} \right)^{\frac{1}{p}}\le K_p^{\frac{1}{p}} {{\Phi }}\left( {{\left( A+B \right)}^{p}} \right)^{\frac{1}{p}},
\end{equation*}
and
\begin{equation*}
{{\Phi }}\left( {{A}^{p}} \right)^{\frac{1}{p}}+{{\Phi }}\left( {{B}^{p}} \right)^{\frac{1}{p}}\le \beta_p+{{\Phi }}\left( {{\left( A+B \right)}^{p}} \right)^{\frac{1}{p}},
\end{equation*}
where $K_p=K(p,m,M)$ is defined as in \eqref{18} and $\beta_p$ is as in \eqref{beta_p}.
In particular, when $p=2,$

\begin{equation*}\label{13}
{{\Phi }}\left( {{A}^{2}} \right)^{\frac{1}{2}}+{{\Phi }}\left( {{B}^{2}} \right)^{\frac{1}{2}}\le \frac{M+m}{2\sqrt{Mm}}{{\Phi }}\left( {{\left( A+B \right)}^{2}} \right)^{\frac{1}{2}},
\end{equation*}
and
\begin{equation*}\label{14}
{{\Phi }}\left( {{A}^{2}} \right)^{\frac{1}{2}}+{{\Phi }}\left( {{B}^{2}} \right)^{\frac{1}{2}}\le \frac{(M-m)^2}{2\left( M+m \right)}+{{\Phi }}\left( {{\left( A+B \right)}^{2}} \right)^{\frac{1}{2}}.
\end{equation*}
\end{corollary}
We conclude this section by presenting the following Minkowski-type inequalities as an application of Corollary \ref{4}. 
\begin{corollary}
Let ${{A}_{1}},\ldots ,{{A}_{k}}$ and ${{B}_{1}},\ldots ,{{B}_{k}}$ be  Hermitian matrices satisfying $mI\le {{A}_{i}},{{B}_{i}}\le MI$ for $i=1,\ldots ,k$ and some scalars $0<m<M$, and let ${{\Phi }_{1}},\ldots ,{{\Phi }_{k}}:{{\mathscr{M}}_{n}}\to {{\mathscr{M}}_{\ell}}$ be  positive linear mappings with $\sum\nolimits_{i=1}^{k}{{{\Phi }_{i}}\left( I \right)}=I$. Then 
\begin{equation}\label{7}
{{\left( \sum\limits_{i=1}^{k}{{{\Phi }_{i}}\left( A_{i}^{2} \right)} \right)}^{\frac{1}{2}}}+{{\left( \sum\limits_{i=1}^{k}{{{\Phi }_{i}}\left( B_{i}^{2} \right)} \right)}^{\frac{1}{2}}}\le \frac{M+m}{2\sqrt{Mm}}{{\left( \sum\limits_{i=1}^{k}{{{\Phi }_{i}}\left( {{\left( {{A}_{i}}+{{B}_{i}} \right)}^{2}} \right)} \right)}^{\frac{1}{2}}},
\end{equation}
and
\begin{equation}\label{15}
{{\left( \sum\limits_{i=1}^{k}{{{\Phi }_{i}}\left( A_{i}^{2} \right)} \right)}^{\frac{1}{2}}}+{{\left( \sum\limits_{i=1}^{k}{{{\Phi }_{i}}\left( B_{i}^{2} \right)} \right)}^{\frac{1}{2}}}\le \frac{(M-m)^2}{2\left( M+m \right)}+{{\left( \sum\limits_{i=1}^{k}{{{\Phi }_{i}}\left( {{\left( {{A}_{i}}+{{B}_{i}} \right)}^{2}} \right)} \right)}^{\frac{1}{2}}}.
\end{equation}
\end{corollary}
\begin{proof}
If ${{A}_{1}},\ldots ,{{A}_{k}}\in {{\mathscr{M}}_{n}}$ are positive matrices, then $A={{A}_{1}}\oplus \cdots \oplus {{A}_{k}}$ is a positive matrix in ${{\mathscr{M}}_{k}}\left( {{\mathscr{M}}_{n}} \right)$. Let the unital positive linear mapping $\Phi :{{\mathscr{M}}_{k}}\left( {{\mathscr{M}}_{n}} \right)\to {{\mathscr{M}}_{\ell}}$ be defined by $\Phi \left( A \right)=\sum\nolimits_{i=1}^{k}{{{\Phi }_{i}}\left( {{A}_{i}} \right)}$. Utilizing Corollary \ref{4}, we obtain the desired inequalities \eqref{7} and \eqref{15}.
\end{proof}

In particular, we have the following.
\begin{corollary}
Let ${{A}_{1}},\ldots ,{{A}_{k}}$ and ${{B}_{1}},\ldots ,{{B}_{k}}$ be  Hermitian matrices satisfying $mI\le {{A}_{i}},{{B}_{i}}\le MI$ for $i=1,\ldots ,k$ and some scalars $0<m<M$, and let ${{w}_{1}},\ldots ,{{w}_{k}}$ be positive scalars satisfying $\sum\nolimits_{i=1}^{k}{{{w}_{i}}}=1$. Then
\begin{equation}\label{8}
{{\left( \sum\limits_{i=1}^{k}{{{w}_{i}}A_{i}^{2}} \right)}^{\frac{1}{2}}}+{{\left( \sum\limits_{i=1}^{k}{{{w}_{i}}B_{i}^{2}} \right)}^{\frac{1}{2}}}\le \frac{M+m}{2\sqrt{Mm}}{{\left( \sum\limits_{i=1}^{k}{{{w}_{i}}{{\left( {{A}_{i}}+{{B}_{i}} \right)}^{2}}} \right)}^{\frac{1}{2}}}
\end{equation}
and
\begin{equation}\label{16}
{{\left( \sum\limits_{i=1}^{k}{{{w}_{i}}A_{i}^{2}} \right)}^{\frac{1}{2}}}+{{\left( \sum\limits_{i=1}^{k}{{{w}_{i}}B_{i}^{2}} \right)}^{\frac{1}{2}}}\le \frac{(M-m)^2}{2\left( M+m \right)}+{{\left( \sum\limits_{i=1}^{k}{{{w}_{i}}{{\left( {{A}_{i}}+{{B}_{i}} \right)}^{2}}} \right)}^{\frac{1}{2}}}.
\end{equation}
\end{corollary}
\begin{proof}
By applying inequalities \eqref{7} and \eqref{15} for positive linear mappings ${{\Phi }_{i}}:{{\mathscr{M}}_{n}}\to {{\mathscr{M}}_{n}}$ determined by ${{\Phi }_{i}}:T\mapsto {{w}_{i}}T$, $i=1,\ldots ,k$, we get \eqref{8} and \eqref{16}.
\end{proof}

\section{A counterexample}
In studying the equivalence of inequalities \eqref{9} and \eqref{10}, we first tried to prove the following inequality:
\begin{equation*}\label{11}
{{\Phi }}\left( {{A}^{-1}} \right)^{2}\le \frac{{{\left( M+m \right)}^{2}}}{4Mm}{{\Phi }}\left( A \right)^{-\frac{1}{2}}\Phi \left( {{A}^{-1}} \right){{\Phi }}\left( A \right)^{-\frac{1}{2}}.
\end{equation*}
This inequality is not true, as noticed by Yamazaki \cite{yamazaki}. To show this, let $A=\left( \begin{matrix}
   x & 1  \\
   1 & 1  \\
\end{matrix} \right)$,  and define
	\[\Phi \left( A \right)\equiv\frac{1}{2}{{U}^{*}}AU+\frac{1}{2}{{V}^{*}}AV,\] 
where $U$ and $V$ are $2\times 2$ unitary matrices. Then $\Phi $ is a unital positive linear map. We set unitary matrices
	\[U=\left( \begin{matrix}
   \cos \alpha  & -\sin \alpha   \\
   \sin \alpha  & \cos \alpha   \\
\end{matrix} \right),\text{ }V=\left( \begin{matrix}
   \cos \beta  & -\sin \beta   \\
   \sin \beta  & \cos \beta   \\
\end{matrix} \right)\] 
with $\alpha ,\beta \in \mathbb{R}$. Then for $x>0$, we have only to check
	\[T\left( x,\alpha ,\beta  \right)\equiv \frac{{{\left( 1+x \right)}^{2}}}{4x}{{\Phi }}\left( A \right)^{-\frac{1}{2}}\Phi \left( {{A}^{-1}} \right){{\Phi }}\left( A \right)^{-\frac{1}{2}}-{{\Phi }}\left( {{A}^{-1}} \right)^{2}\ge 0.\] 
With the help of Mathematica, we get
	\[T\left( 2,\frac{\pi }{3},\frac{\pi }{6} \right)=\left( \begin{matrix}
   -0.0842034 & -0.185577  \\
   -0.185577 & -0.826511  \\
\end{matrix} \right)\] 
and its eigenvalues are $-0.87032$, $-0.0403946$.

\subsection*{Declarations}
\begin{itemize}
\item {\bf{Availability of data and materials}}:  Not applicable
\item {\bf{Competing interests}}: The authors declare that they have no competing interests.
\item {\bf{Funding}}: This research is supported by a grant (JSPS KAKENHI, Grant Number: 16K05257 and 21K03341), awarded to the author S. Furuichi.
\item {\bf{Authors' contributions}}: Authors declare that they have contributed equally to this paper. All authors have read and approved this version.
\item {\bf{Acknowledgments}}: The authors would like to thank the anonymous reviewers for their valuable comments. Their comments have significantly improved the quality of our work.

\end{itemize}

{\tiny \vskip 0.3 true cm }

{\tiny (M. Sababheh) Vice president, Princess Sumaya University For Technology, Al Jubaiha, Amman 11941, Jordan.}

{\tiny \textit{E-mail address:} sababheh@yahoo.com; sababheh@psut.edu.jo}

{\tiny \vskip 0.3 true cm }

{\tiny (H. R. Moradi) Department of Mathematics, Payame Noor University (PNU), P.O. Box 19395-4697, Tehran, Iran.}

{\tiny \textit{E-mail address:} hrmoradi@mshdiau.ac.ir }

{\tiny \vskip 0.3 true cm }

{\tiny (I. H. G\"um\"u\c s) Department of Mathematics, Faculty of Arts and Sciences, Ad\i aman University, Ad\i yaman, Turkey.}

{\tiny \textit{E-mail address:} igumus@adiyaman.edu.tr}

{\tiny \vskip 0.3 true cm }

{\tiny (S. Furuichi) Department of Information Science, College of Humanities and Sciences, Nihon University, 3-25-40, Sakurajyousui,
Setagaya-ku, Tokyo, 156-8550, Japan.}

{\tiny \textit{E-mail address:} furuichi.shigeru@nihon-u.ac.jp}
\end{document}